\documentclass[12pt,a4paper]{amsart}
\usepackage{graphicx}
\usepackage{amsmath,amssymb,amsthm,mathtools}
\usepackage{amsfonts,amsxtra,amsthm,tikz}
\usetikzlibrary{patterns,shapes}
\usepackage[headings]{fullpage}
\usepackage{ulem}

\RequirePackage[colorlinks=false]{hyperref}
\numberwithin{equation}{section}

\newtheorem{theorem}{Theorem}[section]
\newtheorem{proposition}[theorem]{Proposition} 
\newtheorem{lemma}[theorem]{Lemma}              

\theoremstyle{definition}
\newtheorem{definition}[theorem]{Definition}        
\newtheorem{remark}[theorem]{Remark}              

\theoremstyle{remark}
\newtheorem{example}[theorem]{Example}          

\usetikzlibrary{arrows,positioning}
\usetikzlibrary{decorations.markings}
\tikzstyle arrowstyle=[scale=1.5]
\tikzstyle directed=[postaction={decorate,decoration={markings,
    mark=at position .625 with {\arrow[arrowstyle]{stealth}}}}]
\tikzstyle directed2=[postaction={decorate,decoration={markings,
    mark=at position .535 with {\arrow[arrowstyle]{stealth}}}}]

\newcommand\red{\color{red}}
\newcommand\blue{\color{blue}}

\renewcommand\S{\mathcal{S}}

\newcommand\R{\mathbb{R}}
\newcommand\N{\mathbb{N}}
\newcommand\Z{\mathbb{Z}}
\newcommand\bz{\mathbf{z}}

\DeclareMathOperator{\Id}{Id}
\newcommand\ba{\mathbf{a}}
\newcommand\bb{\mathbf{b}}
\newcommand\bc{\mathbf{c}}
\newcommand\bd{\mathbf{d}}

\newcommand\bp{\mathbf{p}}

\newcommand\lu{\mathcal{U}}
\newcommand\ld{\mathcal{D}}
\newcommand\su{\boldsymbol{1}}
\newcommand\sz{\boldsymbol{0}}
\newcommand\tp{(\text{\small$i$}{,}\text{\small$i+1$})}
\newcommand\tq[2]{(\text{\small$#1$}{,}\text{\small$#2$})}
\DeclareMathOperator{\Rev}{Reverse}

\title{Pak-Stanley labeling of the $m$-Catalan hyperplane arrangement}

\author[R.~Duarte]{Rui Duarte} \address{CIDMA and Department of Mathematics, University of Aveiro, 3810-193 Aveiro, Portugal}
\email{rduarte@ua.pt}

\author[A.~Guedes~de~Oliveira]{Ant\'onio Guedes de Oliveira}
\address{CMUP and Department of Mathematics, Faculty of Sciences, University of Porto, 4169-007 Porto, Portugal}
\email{agoliv@fc.up.pt}

\begin{document}

\begin{abstract}
We characterize in simple terms the Pak-Stanley labels $\lambda(R)$ of the regions  $R$ of the $m$-Catalan arrangement. We also propose a simple algorithm that returns $R$ from $\lambda(R)$. Finally, we characterize in close terms the labels of the
relatively bounded regions.
\end{abstract}

\maketitle

\section{Introduction}
In this paper we extend a construction, due to Pak and Stanley \cite{Stan2}, in which every region
$R$ of the Shi arrangement of hyperplanes is labeled with a function $\lambda(R)\in\N^n$,
in fact a \emph{parking function}, and the labeling is known to be bijective \cite{Stan2}, although an explicit inverse (explicitly defining $R$ out of $\lambda(R)$) is
not known (cf.~\cite{DGO}).

More precisely,  we consider the same construction applied to the regions of the \emph{Catalan arrangement}.
We characterize the labels, show that the labeling is still bijective and we explicitly define the inverse.

 Note that the introduction  of this construction by Pak and Stanley
triggered a variety of research projects in several directions.
Among them, we refer to  the introduction of a different bijective function $\alpha$, due to Athanasiadis and Linusson \cite{AthanLin}, such that $\alpha(R)$
(though in general different from $\lambda(R)$) is still a parking function.
We refer also to the generalization of the construction to $m$-Shi arrangements, giving rise to a generalization of parking functions, the
$m$-parking functions \cite{Stan2},  and to a much wider generalization that gives rise to $G$-parking functions, due to Mazin
\cite{Mazin,PosSha}. Among the projects, of course,  some  aim at deepening the study of the method in itself, for example by constructing an algorithm that returns $R$ from
$\lambda(R)$ and by proving that $\lambda$ maps the relatively bounded regions bijectively onto the prime parking functions \cite{DGO}.

In the meantime, we know the sets of labels of the regions of other arrangements labeled by the same construction, like the Ish arrangement \cite{DGO2} and like some arrangements that interpolate between the Shi arrangement and the Ish arrangement \cite{DGO3}. But we did not know what are the labels of the regions of the Catalan arrangement, perhaps the arrangement closest in structure to the Shi arrangement. This is what is done in this paper. More precisely, we characterize here in simple terms the labels $\lambda(R)$ of the regions $R$ of the $m$-Catalan arrangement with the Pak-Stanley labeling. Moreover, we consider the inverse that returns $R$ out of $\lambda(R)$. Finally, we also characterize in close terms the labels of the relatively bounded regions. This is mostly done in a self-contained and elementary way.

The key ingredient is the notion of \emph{center}, that we introduced in the construction of an algorithm that returns $R$ from $\lambda(R)$ in the Shi arrangement \cite{DGO}.
The center characterizes the Ish parking functions \cite{DGO2} and the labels of the regions of the arrangements that
interpolate between the Shi arrangement and the Ish arrangement \cite{DGO3}. Finally, we proved \cite{DGO4} that the centers with a given length occur exactly as often in parking functions as in rook words of the same dimension (rook words
were used for labeling the regions of the Ish arrangement \cite{LRW}). Note that this notion of center is an extension of the same notion as defined in the cited articles.
In this paper, we prove two main properties: first, that the center (in the new sense) may characterize the increasing functions that are componentwise less than the maximum label of the fundamental chamber; second, that the center of $\ba\in\N^n$ does not change if e.g. we reflect $\ba$ in any of the
hyperplanes of equation $x_i=x_j$
(cf. Figure~\ref{2Catalan}, where $n=3$).

We note that the Catalan arrangement was considered by Athanasiadis in the wider context of general Weyl groups \cite{Athan}. The Pak-Stanley labeling, however, was not considered.

\section{Preliminaries}
As usual, set for integers $k,n$ such that $k\leq n$, $[k,n] :=\{k,k+1,\dotsc,n-1,n\}$ and $[n]=[1,n]$.
\begin{definition}
Let  $n,k>0$ and $\ell\geq0$ be integers. We define
the \emph{$n$-dimensional $(k,\ell)$-Coxeter arrangement}  as the  arrangement
formed by the  $(\ell+k+1)\binom{n}{2}$
hyperplanes of $\R^n$ of equation
$$x_i-x_j=a, \text{ with } 1\leq i<j\leq n\,,\ a\in[-\ell,k]\,.$$
\end{definition}

Hence, the $n$-dimensional Shi arrangement is the $(1,0)$-Coxeter arrangement. Remember that
the \emph{regions} of the Shi arrangement, the connected components of the complement of the union of the
hyperplanes, can be bijectively labeled  by using
the Pak-Stanley labeling, defined as follows \cite[ad.]{Stan}.

The region $R_0$, bounded by the hyperplane of equation $x_1-x_n=1$ and by the hyperplanes of equation $x_i-x_{i+1}=0$ for $i=1,\dotsc,n-1$, is labeled with $\su:=(1,1,\dotsc,1)\in\N^n$~\footnote{This region is sometimes labeled with $\sz=(0,\dotsc,0)\in\Z^n$. In this case, usually the parking functions are our parking functions decreased by $\su$.}.
Given two regions, $R_1$ and $R_2$, separated by a unique hyperplane $H$ such that $R_0$ and $R_1$ are on the same side of $H$,
if the equation of $H$ is of form $x_i-x_j=0$ [resp. of form $x_i-x_j=1$] then the label of $R_2$ is equal to the label of $R_1$ except that
the $i^{\rm th}$ coordinate of $R_2$ is the $i^{\rm th}$ coordinate of $R_1$ increased by $1$
[resp. the $j^{\rm th}$ coordinate of $R_2$ is the $j^{\rm th}$ coordinate of $R_1$ increased by $1$].

In general, if the equation of $H$ is of form $x_i-x_j=m$ with $1\leq i<j\leq n$ and $m>0$, then the $j^{\rm th}$ coordinate of $R_1$
is increased by $1$.
If  $1\leq i<j\leq n$ but $m\leq0$, then it is the $i^{\rm th}$ coordinate of $R_1$ that is increased by $1$~\footnote{In fact, Stanley \cite{Stan,Stan2}
considers the opposite situation}.
In here, we prefer to always write the equation in form $x_i-x_j=m$ with $m\geq0$ (and with either $i>j$ or $i<j$). In this way,
 it is always the   $j^{\rm th}$ coordinate that is increased by $1$.
We now consider the same labeling, but applied to the regions of the classical Catalan arrangement,
which is the $(1,1)$-Coxeter arrangement.
In particular, we characterize the labels that occur in the labeling.

Note that, in the case of the Shi arrangement, the set of labels is the set of parking functions \cite{Stan}, which
 are the functions for which the unique weakly increasing rearrangement $(b_1,\dotsc,b_n)$ satisfies $b_i\leq i$ for every $i\in[n]$ (see \cite{Yan} for a recent survey).
But we know more. The $n$-dimensional $m$-Shi arrangement, which is the $(m,m-1)$-Coxeter arrangement,
has $(mn+1)^{n-1}$ regions. The labels of the Pak-Stanley labeling in this case form \cite{Stan} the
\emph{$m$-parking functions}  $\ba$ for which
the unique weakly increasing rearrangement $(b_1,\dotsc,b_n)$ satisfies
 $$b_i\leq m(i-1)+1\,.$$

Finally, we note that the  $n$-dimensional $m$-Catalan arrangement is the $(m,m)$-Coxeter arrangement. The number of regions of this arrangement is  $\frac{(m n+n)!}{(mn+1)!}=n!F(n,m)$, where $F(n,m)=\frac{1}{mn+1}\binom{mn+n}{mn}$ is a \emph{Fuss-Catalan number}.
In this paper we also characterize the Pak-Stanley labels of these regions, and show that the labeling is bijective.
 For an example, see Fig.~\ref{2Catalan}, where the $3$-dimensional
$2$-Catalan is represented with the $72=9!/7!$
regions thus labeled. We will further show that the Pak-Stanley labeling of the regions of the $(k,\ell)$-Coxeter arrangement is bijective if and only if either $k=\ell+1$ (so that the arrangement is the $k$-Shi arrangement) or  $k=\ell$ (where the arrangement is the $k$-Catalan arrangement).

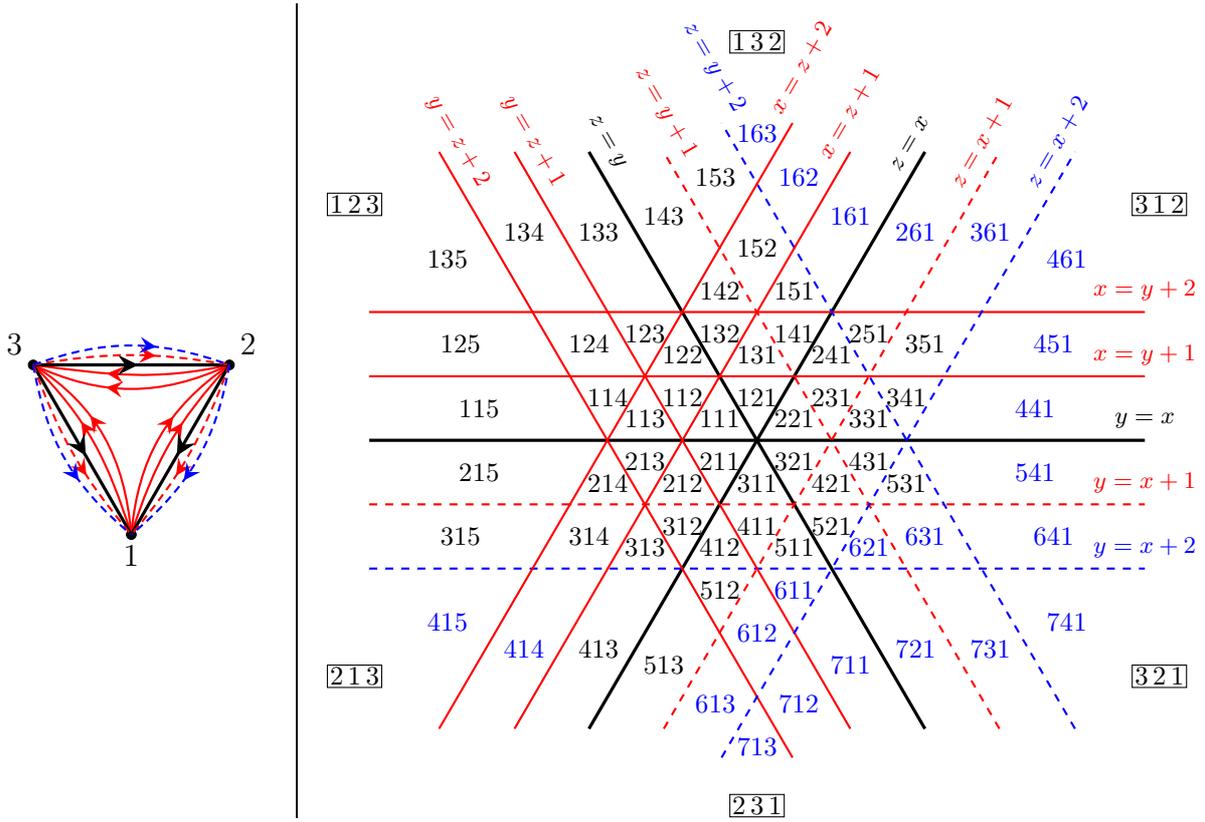
\begin{figure}[t]
\noindent
\begin{tikzpicture}[scale=1.5,baseline=-5.25cm]
\draw [fill] (-.86,.5) node [above left] {$3$} circle [radius=.0425];
\draw [fill] (.86,.5) node [above right] {$2$} circle [radius=.0425];
\draw [fill] (0,-1) node [below] {$1$} circle [radius=.0425];
\draw  [directed, red, thick]  (.86,.5) to [out=192.5,in=352.5] (-.86,.5);
\draw  [directed, red, thick]  (.86,.5) to [out=205,in=335] (-.86,.5);
\draw  [directed, red, thick]  (0,-1) to [out=72.5,in=227.5] (.86,.5);
\draw  [directed, red, thick]  (0,-1) to [out=85,in=215] (.86,.5);
\draw  [directed, red, thick]  (0,-1) to [out=95,in=327.5] (-.86,.5);
\draw  [directed, red, thick]  (0,-1) to [out=107.5,in=315] (-.86,.5);
\draw  [directed2, very thick]  (-.86,.5) -- (.86,.5);
\draw  [directed, red,  densely dashed, thick]  (-.86,.5) to [out=10,in=170] (.86,.5) ;
\draw  [directed, blue,  densely dashed, thick]  (-.86,.5) to [out=20,in=160] (.86,.5) ;
\draw  [directed2, very thick]  (-.86,.5) --  (0,-1) ;
\draw  [directed, red,  densely dashed, thick]  (-.86,.5) to [out=290,in=130] (0,-1) ;
\draw  [directed, blue,  densely dashed, thick]  (-.86,.5) to [out=280,in=140] (0,-1) ;
\draw  [directed2, very thick]  (.86,.5) -- (0,-1);
\draw  [directed, red,  densely dashed, thick]  (.86,.5) to [out=250,in=50] (0,-1) ;
\draw  [directed, blue,  densely dashed, thick]  (.86,.5) to [out=260,in=40] (0,-1) ;
\end{tikzpicture}\hfill\vrule\hfill
\begin{tikzpicture}[scale=2.55]
\draw[ draw=red] (0,-1.334) circle[radius=0.035mm];
\draw[ draw=red] (0,-.667) circle[radius=0.035mm];
\draw[ draw=red] (0,0) circle[radius=0.035mm];
\draw[ draw=red] (0,.666) circle[radius=0.035mm];
\draw[ draw=red] (0,1.333) circle[radius=0.035mm];
\draw[ draw=red] (-1.155,-.667) circle[radius=0.035mm];
\draw[ draw=red] (-1.155,.666) circle[radius=0.035mm];
\draw[ draw=red] (-.963,-.334) circle[radius=0.035mm];
\draw[ draw=red] (-.963,.333) circle[radius=0.035mm];
\draw[ draw=red] (-.770,-.667) circle[radius=0.035mm];
\draw[ draw=red] (-.770,0) circle[radius=0.035mm];
\draw[ draw=red] (-.770,.666) circle[radius=0.035mm];
\draw[ draw=red] (-.578,-.334) circle[radius=0.035mm];
\draw[ draw=red] (-.578,.333) circle[radius=0.035mm];
\draw[ draw=red] (-.385,-.667) circle[radius=0.035mm];
\draw[ draw=red] (-.385,0) circle[radius=0.035mm];
\draw[ draw=red] (-.385,.666) circle[radius=0.035mm];
\draw[ draw=red] (-.193,-1.000) circle[radius=0.035mm];
\draw[ draw=red] (-.193,-.334) circle[radius=0.035mm];
\draw[ draw=red] (-.193,.333) circle[radius=0.035mm];
\draw[ draw=red] (-.193,1.000) circle[radius=0.035mm];
\draw[ draw=red] (.192,-1.000) circle[radius=0.035mm];
\draw[ draw=red] (.192,-.334) circle[radius=0.035mm];
\draw[ draw=red] (.192,.333) circle[radius=0.035mm];
\draw[ draw=red] (.192,1.000) circle[radius=0.035mm];
\draw[ draw=red] (.384,-.667) circle[radius=0.035mm];
\draw[ draw=red] (.384,0) circle[radius=0.035mm];
\draw[ draw=red] (.384,.666) circle[radius=0.035mm];
\draw[ draw=red] (.577,-.334) circle[radius=0.035mm];
\draw[ draw=red] (.577,.333) circle[radius=0.035mm];
\draw[ draw=red] (.769,-.667) circle[radius=0.035mm];
\draw[ draw=red] (.769,0) circle[radius=0.035mm];
\draw[ draw=red] (.769,.666) circle[radius=0.035mm];dashed,blue
\draw[ draw=red] (.962,-.334) circle[radius=0.035mm];
\draw[ draw=red] (.962,.333) circle[radius=0.035mm];
\draw[ draw=red] (1.154,-.667) circle[radius=0.035mm];
\draw[ draw=red] (1.154,.666) circle[radius=0.035mm];

\draw  [very thick] (-2.000,0) -- (2.000,0) node [pos=1.,above,sloped]  {\scriptsize$y=x$};
\draw  [very thick] (-.867,-1.500) -- (.866,1.500) node [pos=1.,above,sloped]  {\scriptsize$z=x$};
\draw  [very thick]  (.866,-1.500) -- (-.867,1.500) node [pos=1.,above,sloped]  {\scriptsize$z=y$};
\draw  [thick,red] (-2.000,.333) -- (2.000,.333) node [pos=1.,above,sloped]  {\scriptsize$x=y+1$};
\draw  [thick,red] (-2.000,.667) -- (2.000,.667) node [pos=1.,above,sloped]  {\scriptsize$x=y+2$};
\draw  [thick,dashed,red] (-2.000,-.333) -- (2.000,-.333) node [pos=1.,above,sloped]  {\scriptsize$y=x+1$};
\draw  [thick,dashed,blue] (-2.000,-.667) -- (2.000,-.667) node [pos=1.,above,sloped]  {\scriptsize$y=x+2$};
\draw  [thick,dashed,red] (-.481,-1.500) -- (1.250,1.500)   node [pos=1,above,sloped]  {\scriptsize$z=x+1$};
\draw  [thick,red] (-1.251,-1.500) -- (.481,1.500)  node [pos=1.05,above,sloped]  {\scriptsize$x=z+1$};
\draw  [thick,dashed,red] (1.250,-1.500) -- (-.482,1.500) node [pos=1.05,above,sloped]  {\scriptsize$z=y+1$};
\draw  [thick,red] (.481,-1.500) -- (-1.251,1.500) node [pos=1.,above,sloped]  {\scriptsize$y=z+1$};
\draw  [thick,dashed,blue] (1.64, -1.5) -- (-0.183, 1.65) node [pos=1.075,above,sloped]  {\scriptsize$z=y+2$};
\draw  [thick,red] (0.183, -1.65) -- (-1.64, 1.5) node [pos=1.,above,sloped]  {\scriptsize$y=z+2$};
\draw  [thick,dashed,blue] (-0.183, -1.65) -- (1.64, 1.5) node [pos=1.,above,sloped]  {\scriptsize$z=x+2$};
\draw  [thick,red] (-1.64, -1.5) -- (0.183, 1.65) node [pos=1.075,above,sloped]  {\scriptsize$x=z+2$};
\node  at (-1.597, 0.944) {\footnotesize $135$};
\node[rectangle, inner sep=1pt, minimum width=3pt, draw=black, color=black]  at (-2.076, 1.227) {\footnotesize $1\,2\,3$};
\node[rectangle, inner sep=1pt, minimum width=3pt, draw=black, color=black]  at (0., 2.072) {\footnotesize $1\,3\,2$};
\node[rectangle, inner sep=1pt, minimum width=3pt, draw=black, color=black]  at (2.075, 1.227) {\footnotesize $3\,1\,2$};
\node[rectangle, inner sep=1pt, minimum width=3pt, draw=black, color=black]  at (2.075, -1.227) {\footnotesize $3\,2\,1$};
\node[rectangle, inner sep=1pt, minimum width=3pt, draw=black, color=black]  at (0., -1.9) {\footnotesize $2\,3\,1$};
\node[rectangle, inner sep=1pt, minimum width=3pt, draw=black, color=black]  at (-2.076, -1.227) {\footnotesize $2\,1\,3$};
\node  at (-1.203,1.083) {\footnotesize $134$};
\node  at (-.818,1.083) {\footnotesize $133$};
\node  at (-.482,1.166) {\footnotesize $143$};
\node  at (-.215,1.370) {\footnotesize $153$};
\node  at (0,1.594) {\blue\footnotesize $163$};
\node  at (0,1.000) {\footnotesize $152$};
\node  at (.214,1.370) {\blue\footnotesize $162$};
\node  at (.481,1.166) {\blue\footnotesize $161$};
\node  at (.817,1.083) {\blue\footnotesize $261$};
\node  at (1.202,1.083) {\blue\footnotesize $361$};
\node  at (1.596, 0.944) {\blue\footnotesize $461$};
\node  at (-1.530, 0.500) {\footnotesize $125$};
\node  at (-.867, 0.500) {\footnotesize $124$};
\node  at (-.578, 0.555) {\footnotesize $123$};
\node  at (-.385, 0.444) {\footnotesize $122$};
\node  at (-.193, 0.555) {\footnotesize $132$};
\node  at (0, 0.444) {\footnotesize $131$};
\node  at (.192, 0.555) {\footnotesize $141$};
\node  at (.384, 0.444) {\footnotesize $241$};
\node  at (.577, 0.555) {\footnotesize $251$};
\node  at (.866, 0.500) {\footnotesize $351$};
\node  at (1.529, 0.500) {\blue\footnotesize $451$};
\node  at (-1.434, 0.166) {\footnotesize $115$};
\node  at (-.770, 0.222) {\footnotesize $114$};,
\node  at (-.578, 0.111) {\footnotesize $113$};
\node  at (-.385, 0.222) {\footnotesize $112$};
\node  at (-.193, 0.111) {\footnotesize $111$};
\node  at (0, 0.222) {\footnotesize $121$};
\node  at (.192, 0.111) {\footnotesize $221$};
\node  at (.384, 0.222) {\footnotesize $231$};
\node  at (.577, 0.111) {\footnotesize $331$};
\node  at (.769, 0.222) {\footnotesize $341$};
\node  at (1.433, 0.166) {\blue\footnotesize $441$};
\node  at (-1.434,-.167) {\footnotesize $215$};
\node  at (-.770,-.223) {\footnotesize $214$};
\node  at (-.578,-.112) {\footnotesize $213$};
\node  at (-.385,-.223) {\footnotesize $212$};
\node  at (-.193,-.112) {\footnotesize $211$};
\node  at (0,-.223) {\footnotesize $311$};
\node  at (.192,-.112) {\footnotesize $321$};
\node  at (.384,-.223) {\footnotesize $421$};
\node  at (.577,-.112) {\footnotesize $431$};
\node  at (.769,-.223) {\footnotesize $531$};
\node  at (1.433,-.167) {\blue\footnotesize $541$};
\node  at (.577,-.556) {\blue\footnotesize $621$};
\node  at (-1.530,-.500) {\footnotesize $315$};
\node  at (-.867,-.500) {\footnotesize $314$};
\node  at (-.578,-.556) {\footnotesize $313$};
\node  at (-.385,-.445) {\footnotesize $312$};
\node  at (-.193,-.556) {\footnotesize $412$};
\node  at (0,-.445) {\footnotesize $411$};
\node  at (.192,-.556) {\footnotesize $511$};
\node  at (.384,-.445) {\footnotesize $521$};
\node at(.866,-.500){\blue\footnotesize $631$};
\node at(1.529,-.500){\blue\footnotesize $641$};
\node at(-1.597,-.945){\blue\footnotesize $415$};
\node at(-1.203,-1.084){\blue\footnotesize $414$};
\node at(-.818,-1.084){\footnotesize $413$};
\node at(-.482,-1.167){\footnotesize $513$};
\node at(-.215,-1.371){\blue\footnotesize $613$};
\node at(0,-1.000){\blue\footnotesize $612$};
\node at(0,-1.595){\blue\footnotesize $713$};
\node at(.214,-1.371){\blue\footnotesize $712$};
\node at(.481,-1.167){\blue\footnotesize $711$};
\node at(.817,-1.084){\blue\footnotesize $721$};
\node at(1.202,-1.084){\blue\footnotesize $731$};
\node at(1.596,-.945){\blue\footnotesize $741$};
\node at (-0.192, 0.778) {\footnotesize $142$};
\node at (0.192, 0.778) {\footnotesize $151$};
\node at (-0.192, -0.778) {\footnotesize $512$};
\node at (0.192, -0.778) {\blue\footnotesize $611$};
\end{tikzpicture}
\caption{Digraph associated with the $3$-dimensional \mbox{$2$-Catalan} arrangement (left), and the
Pak-Stanley labeling of the regions of the arrangement.
In blue, the labels that \emph{are not} $2$-parking functions, and within the rectangles the permutations corresponding to the six chambers (right).}
\label{2Catalan}
\end{figure}

\subsection{The chamber of a permutation}
Note that every $(k,\ell)$-Coxeter arrangement is a refinement of the \emph{braid arrangement}, that consists of the hyperplanes of equation $x_i=x_j$ for $1\leq i<j\leq n$. A region $\widetilde{R}$ of the braid arrangement is the set of points of coordinates
$(x_1,\dotsc,x_n)$ that satisfy
\begin{equation}
x_{\pi_1}>x_{\pi_2}>\dotsb>x_{\pi_n}
\label{eq.angle}\end{equation}
for a given (obviously unique) permutation $\pi=(\pi_1,\dotsc,\pi_n)\in\S_n$. We call $\widetilde{R}$
\emph{the chamber of $\pi$}.

For $\ba\in\N^n$ and for a permutation of $[n]$, $\pi\in\S_n$, let
\begin{align*}
&\ld(\ba)=\big(|\{j<i\mid a_j<a_i\}|\big)_{1\leq i\leq n}\,,\\
&\lu(\ba)=\big(|\{j<i\mid a_j>a_i\}|\big)_{1\leq i\leq n}\,,\\
\shortintertext{and}
&I(\pi)=\big(|\{j<\pi^{-1}(i)\mid \pi_j>i\}|\big)_{1\leq i\leq n}\\
&\hphantom{I(\pi)}\in[0,n-1]\times[0,n-2]\times\dotsb\times\{0\}\,.
\shortintertext{$I(\pi)$ is \emph{the inversion table} of $\pi$ \cite[p.36]{StanB}. Finally note that}
&I(\pi)=\lu(\pi)\circ\pi^{-1}\,.
\end{align*}

In any $(k,\ell)$-Coxeter arrangement and in every chamber of a permutation $\pi\in\S_n$ there is a unique region $r_\pi$
whose border contains the line defined by $x_1=x_2=\dotsb=x_n$.
The Pak-Stanley label of this region is $\mu(\pi):=\su+I(\pi)$, where
$I(\pi)$ is the inversion table of $\pi$, defined above \cite[ad.]{Stan2}.
Since the hyperplanes that divide the chamber do not separate the points of $R_0$ from the points of $r_\pi$,
the label $\mu(\pi)$ is \emph{minimal} in the chamber. More precisely, the label of any other region in the chamber is  $\mu(\pi)$ increased by $e_j$ (the vector with all coordinates zero but the $j$th, equal to $1$) for every
plane of equation $x_i-x_j=m$ with $m>0$ that separates the region from $r_\pi$.
In particular, this label is componentwise greater than $\mu(\pi)$.

For any positive integer $q$ and every $\pi\in\S_n$, we may choose a point $P_q$ in the chamber of $\pi$ such that
{$$x_{\pi_1}>x_{\pi_2}+q>x_{\pi_2}> x_{\pi_3}+q>x_{\pi_3}>\dotsb>x_{\pi_n}+q> x_{\pi_n}\,.$$}\\[-30pt]

\noindent
Hence,  for sufficiently large $q$, $P_q$ is separated from $r_\pi$ by all hyperplanes of form
$x_{\pi_i}-x_{\pi_j}=m$ with $1\leq i<j\leq n$ and either with $m\in[k]$ if $\pi_i<\pi_j$ or
with $m\in[\ell]$ if $\pi_j<\pi_i$.
This means that the region $R_\pi$ of the $(k,\ell)$-Coxeter arrangement that contains $P_q$ is labeled with
\begin{align}
M(\pi,k,\ell):=&\mu(\pi)+k(\ld(\pi)\circ\pi^{-1})+\ell(\lu(\pi)\circ\pi^{-1})\nonumber\\
=&\su+k(\ld(\pi)\circ\pi^{-1})+(\ell+1)(\lu(\pi)\circ\pi^{-1})\,.\label{max.def}
\end{align}
This label is  componentwise greater than any other label in the chamber, and, being $k>0$ and $\ell\geq0$, it is indeed maximal in the set of labels.
Note that, according to Mazin \cite{Mazin}, the set $\mathcal{L}$ of labels of the regions of the $(k,\ell)$-Coxeter arrangement  is the set of $G$-parking functions for the multi-digraph $G=(V,A)$ with set of vertices
$V=[n]$ and with $k$ arcs $(i,j)$ and $\ell+1$ arcs $(j,i)$ for every $1\leq i<j\leq n$ (Cf. Figure~\ref{2Catalan}).
Remember that a function $f\colon[n]\to\N_0$ is a \emph{$G$-parking function} if for every non-empty subset $I\subseteq[n]$ there exists a vertex $v\in I$ such that the number of elements $w\notin I$ such that $(v,w)\in A$, counted with multiplicity, is greater than  $f(i)$.
Hence,  by definition of $G$-parking function, $\mathcal{L}$ is a \emph{lower set} in the componentwise order, i.e.,  if $\ba\in\mathcal{L}$
 and $\bb\in\N^n$ is componentwise less that $\ba$, then  $\bb\in\mathcal{L}$.

\subsection{The fundamental chamber}
The chamber of the identity is \emph{the fundamental chamber}~\footnote{The fundamental chamber is also referred to as the \emph{dominant chamber} in the literature.}, which is hence the set of points
$P_{(x_1,\dotsc,x_n)}$ such that
\begin{equation}x_1>x_2>\dotsb>x_n\,.\label{rf.eq}\end{equation}
Note that the only hyperplanes that cross the fundamental chamber are those of equation $x_i-x_j=m$ with
$i<j$ and $m>0$ since  $x_i-x_j>0$ if $i>j$ for every point $P_{(x_1,\dotsc,x_n)}$ in the chamber.
Hence all the fundamental chambers of the $(k,\ell)$-arrangement for $k$ fixed
coincide, even in terms of the regions
that divide the chamber. These regions may be defined by considering the refinement of \eqref{rf.eq} of form
$$x_2+i_1\!+\!1\!>\!x_1\!>\!x_2+i_1\!>\!x_3+i_2\!+\!1\!>\!x_2\!>\!x_3+i_2>\dotsb>x_{n-1}\!>\!x_n+i_{n-1}$$
where $0\leq i_1,i_2,\dotsc,i_{n-1}\leq k$ but  where the term $x_j+i_{j-1}+1$ is omitted for every $j\in[n-1]$
such that $i_j=k$. Such a region is hence labeled
$$(1,1+i_1,1+i_1+i_2,\dotsc,1+i_1+\dotsb+i_{n-1})\,,$$
In particular,
the maximal label is
\begin{equation}
M(\Id_{[n]},k,\ell)=\su+k(\Id_{[n]}-\su)\,.\label{max.eq}
\end{equation}
(cf. \eqref{max.def} since $\ld(\Id_{[n]})=\Id_{[n]}-\su$ and $\lu(\Id_{[n]})=\sz$),
which means that the labels of the regions in the fundamental chamber are exactly the weakly increasing vectors $(a_1\dotsc,a_n)\in\N^n$ for which
\begin{equation}
a_i\leq 1+k(i-1)\,,\ \text{for every $i\in[n]$}\,.\label{fchamber.eq}
\end{equation}

\begin{figure}[h]
\noindent
\begin{tikzpicture}[scale=.75]
\draw[step=1cm,gray,very thin] (0,0) grid (7,3);
\draw[thin] (0,0) -- (7,0) -- (7,3) -- (0,3) -- cycle;
\draw [fill] (1,0) node [above left, inner sep=0pt, minimum width=0pt] {\scriptsize$(1,0)$} circle [radius=.0725] ;
\draw [fill] (1,1) node {} circle [radius=.0225];
\draw [fill] (3,1) node [above left, inner sep=0pt, minimum width=0pt] {\scriptsize$(3,1)$} circle [radius=.0725] ;
\draw [fill] (3,2) node {} circle [radius=.0225];
\draw [fill] (5,2) node [above left, inner sep=0pt, minimum width=0pt] {\scriptsize$(5,2)$} circle [radius=.0725] ;
\draw [fill] (5,3) node {} circle [radius=.0225];
\draw [fill] (7,3) node [above left, inner sep=0pt, minimum width=0pt] {\scriptsize$(kn+1,n)$} circle [radius=.0725] ;
\draw [ultra thick]  (0,0) -- (1,0) -- (1,1) -- (3,1) -- (3,2) -- (5,2) -- (5,3) -- (7,3);
\draw [very thick,red, dashed] (1,0) -- (7,3);
\end{tikzpicture}
\caption{Dyck path representation of $135$ ($n=3$ and $k=2$)}
\label{Dycka}
\end{figure}
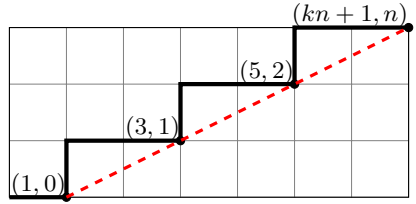

We may represent $M(\Id_{[n]},k,k)=:(m_1,\dotsc,m_n)$ as a Dyck path through the points $(m_i,i-1)$ (aligned
in the diagonal $d$ that joins $(1,0)$ to $(kn+1,n)$) and $(m_i,i)$ for every $i\in[n]$, and ending at
$(kn+1,n)$ (see Figure~\ref{Dycka}). Hence, the labels of the regions of the fundamental chamber correspond to the Dyck paths weakly \emph{above} or \emph{to the left of} $d$.
By  \cite[Theorem 10.4.5]{Krat}, the number of regions is the \emph{Fuss-Catalan number}
$$F(n,k)=\frac{1}{n}\binom{kn+n}{n-1}=\frac{1}{kn+1}\binom{kn+n}{kn}\,.$$

\subsection{The case $k=\ell$}
Let $\pi\in\S_n$, consider the action of $\pi$ on $\R^n$ defined by
\begin{align*}
&\pi(P)=\pi(x_1,x_2,\dotsc,x_n) :=(x_{\pi^{-1}(1)},x_{\pi^{-1}(2)},\dotsc,x_{\pi^{-1}(n)})\\
\shortintertext{and note that the fundamental chamber of any $(k,\ell)$-Coxeter arrangement is mapped onto the corresponding chamber of $\pi$, since}
&x_i=(x_{\pi^{-1}(1)},x_{\pi^{-1}(2)},\dotsc,x_{\pi^{-1}(n)})_{\pi(i)}\,.
\end{align*}
For the same reason, if some point $P_{(x_1,\dotsc,x_n)}$ of the fundamental chamber is contained in the set  defined by
$x_i-x_j<m$ [resp., $x_i-x_j=m$, $x_i-x_j>m$] ($1\leq i<j\leq n, 0<m\leq k$) then $\pi(P)$  is contained in the set defined by
$x_{\pi(i)}-x_{\pi(j)}<m$ [resp., $x_{\pi(i)}-x_{\pi(j)}=m$, $x_{\pi(i)}-x_{\pi(j)}>m$] and vice-versa.
If $k=\ell$,  then $x_{\pi(i)}-x_{\pi(j)}=m$ is again a hyperplane of the $(k,\ell)$-Coxeter arrangement. Hence,
$\pi$  bijectively maps the regions of the fundamental chamber onto the regions of the chamber of $\pi$, and
thus the number of regions of the Catalan arrangement is
$$n!\,F(n,k)\,.$$
This also means that, if $\ba$ is the Pak-Stanley label of the region that contains $P$ and $\bb$ is the Pak-Stanley label of the region that contains $\pi(P)$, then
$$\bb=I(\pi)+\ba\circ\pi^{-1}\,.$$
In particular, the maximal label is
\begin{align*}
&M(\pi,k,k)=I(\pi)+M(\Id_{[n]},k,k)\circ\pi^{-1}\\
&\hphantom{M(\pi,k,k)}=I(\pi)+\big(\su+k(\Id_{[n]}-\su)\big)\circ\pi^{-1}\\
&\hphantom{M(\pi,k,k)}=\mu(\pi)+k(\pi^{-1}-\su)\,.\footnotemark
\end{align*}
\footnotetext{A similar calculation shows that in the $m$-Shi arrangement the maximal label of the chamber of
$\pi$ is $$M(\pi,m,m-1)=\su+m(\pi^{-1}-\su)\,.$$
Since the set $\mathcal{L}$ of labels of the regions of the $(k,\ell)$-Coxeter arrangement  is a lower set, $\mathcal{L}$ is the set of $m$-parking functions.}

\begin{example} The maximal label of the fundamental chamber in Figure~\ref{2Catalan} is
\begin{align*}
&135=111+2(123-111)\,.\\
\shortintertext{The points of the region labeled with $112$ verify}
&x_2+1>x_1>x_3+1>x_2>x_3\\
\shortintertext{and whence the region is the intersection of the half-spaces defined by}
&\negthickspace\begin{array}{lcr}x_1&<&x_2+1\\x_1&>&x_3+1\\x_2&<&x_3+1\end{array}\\
\shortintertext{whereas e.g. the region labeled $125$ verifies (note that $k=2$)}
&x_2+2>x_1>x_2+1>x_2>x_3+2\,.
\shortintertext{Let $\pi=3\,1\,2$ and $\bb$ be the label of the region in the chamber of $\pi$ and in the $\S_3$-orbit of the region labeled $125$.
Then, in fact,}
&\bb=110+125\circ2\,3\,1\\
&\hphantom{\bb}=110+251\\
&\hphantom{\bb}=361\,.
\end{align*}
\end{example}

\subsection{The cases $\ell\neq k\neq\ell+1$}
If $k=\ell+1$ we are in the presence of the $k$-Shi arrangement, where the Pak-Stanley labeling \emph{bijectively} labels the regions  with $k$-parking functions.
In the sequel, we prove that the same happens if $k=\ell$. More precisely, we prove that in this case the Pak-Stanley labeling is also injective (and the image is the set of $k$-Catalan functions).
In order to prove that the same does not happen when $\ell\neq k\neq\ell+1$, we distinguish two cases, when $k>\ell+1$ and when $k<\ell$.

In the first case, we consider the points in dimension $n$
$$P_{(0.8,k+0.3,0,-0.4/n,-0.5/n,\dotsc,-0.1)} \hfill\text{ and }\hfill
Q_{(1.2,k-0.3,0,-0.4/n,-0.5/n,\dotsc,-0.1)}\,.$$
Their position differs exactly in relation with the hyperplanes of equation $x_1-x_3=1$ and $x_2-x_3=k$, and so, on the one hand, they are in different regions of the arrangement, and on the other hand,
the two regions are labeled with the same label (V. Figure~\ref{ninj.fig}).

If $k<\ell$, then we take
$$P_{(-0.2,\ell-0.3,0,-0.4/n,-0.5/n,\dotsc,-0.1)} \hfill\text{ and }\hfill
Q_{(0.2,\ell+0.3,0,-0.4/n,-0.5/n,\dotsc,-0.1)}\,.$$
These points differ exactly on the position relative to the hyperplanes of equation $x_3=x_1$ and $x_2-x_1=\ell$, and again the two corresponding regions receive the same label.

Note that the labels of the $(k,\ell)$-Coxeter arrangement are the reverse of the labels of the $(\ell+1,k-1)$-Coxeter arrangement (cf. Figure~\ref{ninj.fig}). To see this, since the corresponding sets of labels are lower sets, it is sufficient to
see that the maximal labels are reverse to each other.
But, if $\pi\in\S_n$ and $\rho^{-1}=\Rev(\pi^{-1})$,  then
$\ld(\rho)=\lu(\pi)$ and $\lu(\rho)=\ld(\pi)$, and then, according to \eqref{max.def},
\begin{align*}
M(\pi,k,\ell)&=\su+k(\ld(\pi)\circ\pi^{-1})+(\ell+1)(\lu(\pi)\circ\pi^{-1})\\
M(\rho,\ell+1,k-1)&=\su+(\ell+1)(\ld(\rho)\circ\rho^{-1})+k(\lu(\rho)\circ\rho^{-1})\\
&=\Rev\big(M(\pi,k,\ell)\big)\,.
\end{align*}

\renewcommand{\ULthickness}{1.pt}
\ULdepth=0.3ex
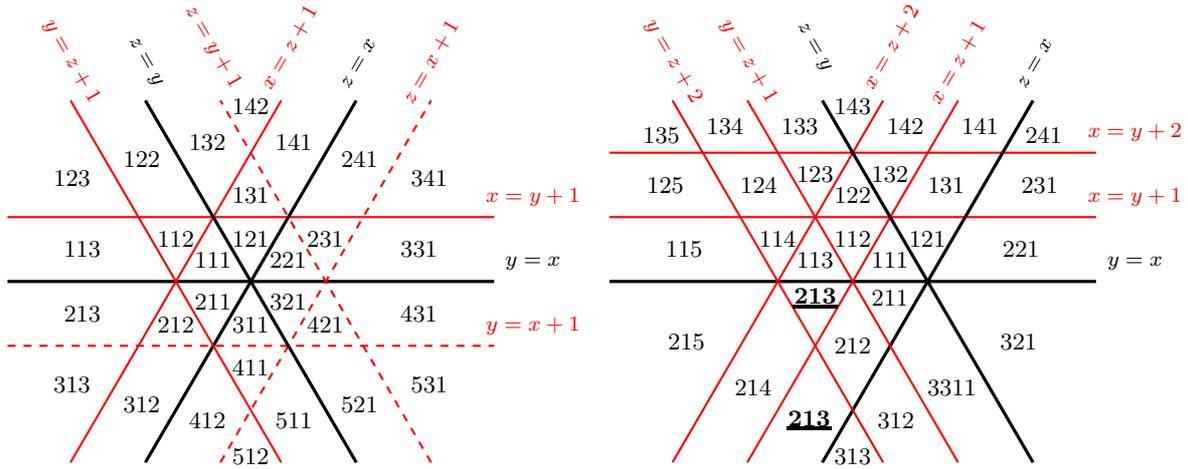
\begin{figure}[ht]
\noindent
\strut\hfill
\begin{tikzpicture}[scale=1.6]
\draw  [very thick] (-2.000,0) -- (2.000,0) node [pos=1.08,above,sloped]  {\tiny$y=x$};
\draw  [thick,red] (-2.000,0.5333) -- (2.000,0.5333) node [pos=1.08,above,sloped]  {\tiny$x=y+1$};
\draw  [thick,dashed,red] (-2.000,-0.5333) -- (2.000,-0.5333)node [pos=1.08,above,sloped]  {\tiny$y=x+1$};
\draw  [very thick] (-.867,-1.500) -- (.867,1.500) node [pos=1.08,above,sloped]  {\tiny$z=x$};
\draw  [thick,dashed,red] (-0.2502,-1.500)  -- (1.482,1.500)  node [pos=1.08,above,sloped]  {\tiny$z=x+1$};
\draw  [thick,red] (-1.482,-1.500) -- (0.2502,1.500) node [pos=1.12,above,sloped]  {\tiny$x=z+1$};
\draw  [very thick]  (.867,-1.500) -- (-.867,1.500) node [pos=1.08,above,sloped]  {\tiny$z=y$};
\draw  [thick,dashed,red] (1.482,-1.500) -- (-0.2502,1.500) node [pos=1.12,above,sloped]  {\tiny$z=y+1$};
\draw  [thick,red] (0.2502,-1.500) -- (-1.482,1.500) node [pos=1.08,above,sloped]  {\tiny$y=z+1$};

\node at(-1.469,.855){\scriptsize$123$};
\node at(-.895,1.016){\scriptsize$122$};
\node at(-.357,1.150){\scriptsize$132$};
\node at(0,1.45){\scriptsize$142$};
\node at(.356,1.150){\scriptsize$141$};
\node at(.894,1.016){\scriptsize$241$};
\node at(1.468,.855){\scriptsize$341$};
\node at(0,.711){\scriptsize$131$};
\node at(-1.385,.266){\scriptsize$113$};
\node at(-.616,.355){\scriptsize$112$};
\node at(-.308,.177){\scriptsize$111$};
\node at(0,.355){\scriptsize$121$};
\node at(.307,.177){\scriptsize$221$};
\node at(.615,.355){\scriptsize$231$};
\node at(1.384,.266){\scriptsize$331$};
\node at(-1.385,-.267){\scriptsize$213$};
\node at(-.616,-.356){\scriptsize$212$};
\node at(-.308,-.168){\scriptsize$211$};
\node at(0,-.356){\scriptsize$311$};
\node at(.307,-.168){\scriptsize$321$};
\node at(.615,-.356){\scriptsize$421$};
\node at(1.384,-.267){\scriptsize$431$};
\node at(-1.469,-.856){\scriptsize$313$};
\node at(-.895,-1.017){\scriptsize$312$};
\node at(-.357,-1.150){\scriptsize$412$};
\node at(0,-1.45){\scriptsize$512$};
\node at(0,-.712){\scriptsize$411$};
\node at(.356,-1.150){\scriptsize$511$};
\node at(.894,-1.017){\scriptsize$521$};
\node at(1.469, -0.8556){\scriptsize$531$};
\end{tikzpicture}
\hfill
\begin{tikzpicture}[scale=1.6]
\draw  [very thick] (-2.000,0) -- (2.000,0) node [pos=1.08,above,sloped]  {\tiny$y=x$};
\draw  [thick,red] (-2.000,0.5333) -- (2.000,0.5333) node [pos=1.08,above,sloped]  {\tiny$x=y+1$};
\draw  [thick, red] (-2.000,1.067) -- (2.000,1.067) node [pos=1.08,above,sloped]  {\tiny$x=y+2$};
\draw  [very thick] (-0.2502,-1.500) -- (1.482,1.500)  node [pos=1.08,above,sloped]  {\tiny$z=x$};
\draw  [thick, red] (-1.482,-1.500) -- (0.2502,1.500) node [pos=1.12,above,sloped]  {\tiny$x=z+2$};
\draw  [thick, red] (-.867,-1.500) -- (.867,1.500) node [pos=1.08,above,sloped]  {\tiny$x=z+1$};
\draw  [very thick]  (1.482,-1.500) -- (-0.2502,1.500) node [pos=1.12,above,sloped]  {\tiny$z=y$};
\draw  [thick,red] (0.2502,-1.500) -- (-1.482,1.500) node [pos=1.08,above,sloped]  {\tiny$y=z+2$};
\draw  [thick,red] (0.8660,-1.500) -- (-0.8660,1.500) node [pos=1.08,above,sloped]  {\tiny$y=z+1$};

\node at(-1.572,1.211){ \scriptsize$135$};
\node at(-1.049,1.283){ \scriptsize$134$};
\node at(-.434,1.283){ \scriptsize$133$};
\node at(0,1.45){ \scriptsize$143$};
\node at(.433,1.283){ \scriptsize$142$};
\node at(1.048,1.283){ \scriptsize$141$};
\node at(1.571,1.211){ \scriptsize$241$};
\node at(-1.539,.800){ \scriptsize$125$};
\node at(-.770,.800){ \scriptsize$124$};
\node at(-.308,.888){ \scriptsize$123$};
\node at(0,.711){ \scriptsize$122$};
\node at(.307,.888){ \scriptsize$132$};
\node at(.769,.800){ \scriptsize$131$};
\node at(1.538,.800){ \scriptsize$231$};
\node at(-1.385,.266){ \scriptsize$115$};
\node at(-.616,.355){ \scriptsize$114$};
\node at(-.308,.177){ \scriptsize$113$};
\node at(0,.355){ \scriptsize$112$};
\node at(.307,.177){ \scriptsize$111$};
\node at(.615,.355){ \scriptsize$121$};
\node at(1.384,.266){ \scriptsize$221$};
\node at(-1.366,-.500){ \scriptsize$215$};
\node at(-.818,-.884){ \scriptsize$214$};
\node at(-.308,-.138){\uline{\scriptsize$\mathbf{213}$}};
\node at(-.357,-1.150){\uline{\scriptsize$\mathbf{213}$}};
\node at(0,-1.45){ \scriptsize$313$};
\node at(.307,-.138){ \scriptsize$211$};
\node at(0,-.534){ \scriptsize$212$};
\node at(.356,-1.150){ \scriptsize$312$};
\node at(.817,-.884){ \scriptsize$3311$};
\node at(1.365,-.500){ \scriptsize$321$};
\end{tikzpicture}
\hfill\strut
\caption{Pak-Sanley labeling of the $3$-dimensional Catalan arrangement (left), and of the $(2,0)$-Coxeter arrangement (right). Same number of labels ($30$) ---the labels of the two arrangements are reverse from each other--- but different number of regions ($30$ and $31$, resp.)}
\label{ninj.fig}
\end{figure}

\section{Pak-Stanley labels of the $m$-Catalan arrangement}
\noindent
\begin{definition}\label{def.crc}
Let $\ba=(a_1,\dotsc,a_n)\in\N^n$ and consider an integer $p$.
The \emph{$p$-center of $\ba$}, $Z_p(\ba)$, is the largest subset $X = \{ x_1, \ldots, x_q\}$ of $[n]$ with
$1\leq x_q<x_{q-1}<\dotsb<x_1\leq n$
and with the property that $a_{x_j} \leq  p+j$ for every $j \in [q]$~\footnote{The \emph{$0$-center} in here is precisely the
\emph{center} as we have introduced it before \cite{DGO2, DGO3}, and previously in reversed form
\cite{DGO}.}.  Clearly,  if $p<0$, then $Z_p(\ba)=\varnothing$.
Let $z_p(\ba)=|Z_p(\ba)|$.
Note that since, by definition,  $Z_0(\ba)\subseteq Z_1(\ba)\subseteq\dotsi\subseteq Z_p(\ba)$,
$z_0(\ba)\leq z_1(\ba)\leq\dotsb\leq z_p(\ba)$.
If $\ba$ is a label of the $m$-Catalan arrangement in dimension $n$, we define \emph{center of $\ba$} to be
$$\bz(\ba):=\big(z_0(\ba),z_1(\ba),\dotsc,z_{(n-1)m}(\ba)\big)\,.$$
 Hence, in particular, $\bz(\ba)$ is weakly increasing. See Figure~\ref{Dycka} for one interpretation of the concept of \emph{center} in this way.
\end{definition}

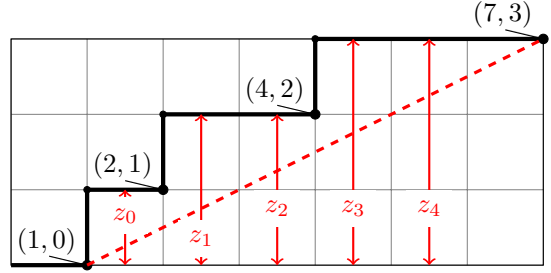
\begin{figure}[ht]
\noindent
\begin{tikzpicture}[scale=1.]
\draw[step=1cm,gray, very thin] (0,0) grid (7,3);
\draw[thin] (0,0) -- (7,0) -- (7,3) -- (0,3) -- cycle;
\draw [fill] (1,0) node [above left] {\footnotesize$(1,0)$} circle [radius=.0625] -- (.5,.12);
\draw [fill] (1,1) node {} circle [radius=.0425];
\draw [fill] (2,1) node [above left] {\footnotesize$(2,1)$} circle [radius=.0625] -- (1.5,1.12);
\draw [fill] (2,2) node {} circle [radius=.0425];
\draw [fill] (4,2) node [above left] {\footnotesize$(4,2)$} circle [radius=.0625] -- (3.5,2.12);
\draw [fill] (4,3) node {} circle [radius=.0425];
\draw [fill] (7,3) node [above left] {\footnotesize$(7,3)$} circle [radius=.0625] -- (6.5,3.12);
\draw [ultra thick]  (0,0) -- (1,0) -- (1,1) -- (2,1) -- (2,2) -- (4,2) -- (4,3) -- (7,3);
\draw [very thick,red, dashed] (1,0) -- (7,3);
\draw [thick, red, <->] (1.5,0) -- node [red,  fill=white, above=-0.1cm] {\footnotesize$z_0$}(1.5,1);
\draw [thick, red, <->] (2.5,0) -- node [red,  fill=white, above=-0.9cm] {\footnotesize$z_1$}(2.5,2);
\draw [thick, red, <->] (3.5,0) -- node [red,  fill=white, above=-0.5cm] {\footnotesize$z_2$}(3.5,2);
\draw [thick, red, <->] (4.5,0) -- node [red,  fill=white, above=-1cm] {\footnotesize$z_3$}(4.5,3);
\draw [thick, red, <->] (5.5,0) -- node [red,  fill=white, above=-1cm] {\footnotesize$z_4$}(5.5,3);
\end{tikzpicture}\hfill
\caption{Dyck path representation of $124$ ---note that $\bz(124)=(1,2,2,3,3)$}
\label{Dycka}
\end{figure}

We present now our Key Lemma.
 \begin{lemma}
Let $\pi,\rho\in\S_n$ differ by an adjacent transposition, say $\rho=\tp\circ\pi$ for some $i\in[0,n-1]$,
and let $\bb=(b_1,\dotsc,b_n)$ and $\bc=(c_1,\dotsc,c_n)$ be the Pak-Stanley labels of the two images, by $\pi$ and by
$\rho$, respectively,  of the same region $R$ of the fundamental chamber of the $m$-Catalan arrangement. Then
$$\bc=\bb\circ\tp+{}\begin{cases}e_{i},&\text{ if $b_i\leq b_{i+1}$}\\(-e_{i+1}),&\text{ otherwise.}\end{cases}\,.$$
\label{keylemma}
\end{lemma}
\begin{proof}
Let $\ba$ be the label of $R$ and note that $\bb=I(\pi)+\ba\circ\pi^{-1}$ and $\bc=I(\rho)+\ba\circ\rho^{-1}$.
If $\pi^{-1}(i)<\pi^{-1}(i+1)$, then by definition $I(\pi)_i\leq I(\pi)_{i+1}$, and $a_{\pi^{-1}(i)}\leq a_{\pi^{-1}(i+1)}$, since
$\ba$ is weakly increasing; so, if $\pi^{-1}(i)<\pi^{-1}(i+1)$, then $b_i\leq b_{i+1}$.
On the other hand, if $\pi^{-1}(i)>\pi^{-1}(i+1)$, then $I(\pi)_i\geq I(\pi)_{i+1}+1$ and $a_{\pi^{-1}(i)}\geq a_{\pi^{-1}(i+1)}$ and so $b_i> b_{i+1}$.
Hence,
$$b_i\leq b_{i+1} \iff \pi^{-1}(i)<\pi^{-1}(i+1)\,.$$

Now, suppose that $b_i\leq b_{i+1}$ and that we swap $i$ and $i+1$ in $\pi$; then $I(\rho)_{i+1}=I(\pi)_i$ but
$I(\rho)_i=I(\pi)_{i+1}+1$ since $i+1$ now precedes $i$ in $\rho$. When $b_i>b_{i+1}$, $I(\rho)_i=I(\pi)_{i+1}$
whereas $I(\rho)_{i+1}=I(\pi)_i-1$.  Hence
    $$I(\rho)=(I(\pi)\circ\tp)+{}\begin{cases}e_{i},&\text{ if $b_i\leq b_{i+1}$}\\(-e_{i+1}),&\text{ otherwise.}\end{cases}$$
We finish the proof by noting that $\rho^{-1}=\pi^{-1}\circ\tp$.
\end{proof}

We now introduce a new function, defined by $\bp(\ba)_i$ being the minimal position of $i$ in the sequence of centers of $\ba$.
\begin{definition} For every $\ba\in\N^n$,
let $\bp(\ba)\in\N^n$ be defined by $\bp(\ba)_i=j$ if $i\in Z_j\setminus Z_{j-1}$.
\end{definition}

\begin{proposition}\label{Kprop}
Let $\pi\in\S_n$, let $\ba\in\N^n$ be the label of a region $R$  in the fundamental chamber of the $m$-Catalan arrangement, let $\bb$ be the label of the region $\pi(R)$,
and let $p\in\{0,1,\dotsc,(m-1)n\}$. Then
\begin{enumerate}
\item $\ \forall i\in[n]\,,\quad i\in Z_p(\ba) \iff \pi(i)\in Z_p(\bb)\,$;\\[-10pt]
\item $\ \forall i,j\in[n]\,, i,j\in Z_{p}(\ba) \setminus Z_{p-1}(\ba) \implies a_i=a_j \,$;\\[-10pt] 
\item $\ \bb-\bp(\bb)=\su+I(\pi)=\mu(\bb)\,$.
\end{enumerate}
\end{proposition}

\begin{proof}
For the first part, it is sufficient to prove that, if $\pi$, $\rho$, $\bb$ and $\bc$ are as described in Lemma~\ref{keylemma},
then $\pi(i)\in Z_j(\bb)$ if and only if $\rho(i)\in Z_j(\bc)$. This fact follows immediately from Lemma~\ref{keylemma}.
In fact, we may consider only the case where, for some $j\in[n]$, $\pi=(j,j+1)$. Then,  by Lemma~\ref{keylemma},
\begin{align*}\bb&=\ba\circ(j,j+1)+{}\begin{cases}e_{i},&\text{ if $a_i\leq a_{i+1};$}\\(-e_{i+1}),&\text{ if $a_i> a_{i+1}.$}\end{cases}\\
\arraycolsep=1.7pt
&=\left\{\begin{array}{rcccll}
(a_1,\dotsc,a_{j-1},&a_{j+1}+1&,&a_j&,a_{j+2},\dotsc,a_n)\,,& \text{\quad if $a_i\leq a_{i+1};$}\\[5pt]
(a_1,\dotsc,a_{j-1},&a_{j+1}&,&a_j-1&,a_{j+2},\dotsc,a_n)\,,&\text{\quad if $a_i> a_{i+1}.$}
\end{array}\right.\end{align*}
For the second part, note that $\ba$ is weakly increasing, and so $i\in Z_{a_i-1}(\ba)$ but $i\notin Z_{a_i-2}(\ba)$,
and thus $\bp=\ba-\su$.
Hence,   for any $\sigma\in\S_n$, if
$\bd$ is the Pak-Stanley label of $\sigma(R)$, then $\bp(\bd)=\ba\circ\sigma^{-1}-\su$, which proves the third part.
\end{proof}

\begin{remark}
Note that $I$, the inversion table, defines a \emph{bijection}
from $\S_n$ onto
$$[0,n-1]\times[0,n-2]\times\dotsb\times\{0\}$$
(cf. \cite[p.~36]{StanB};
for a convenient method for obtaining $\pi$ from $I(\pi)$ see S-parking \cite{DGO}). Hence, $\bb$ defines $\pi$ under the conditions of Proposition~\ref{Kprop}, and of course $\pi$ defines $\bb$.
In particular, the number of different labels of the regions of the same orbit under the action of $\S_n$ is $n!$.
\label{bij.rem}
\end{remark}

\begin{example}
In the $3$-dimensional $2$-Catalan arrangement represented in Figure~\ref{2Catalan}, the $\S_3$-orbit of $124$ is
$\{124,152,351,631,612,314\}$.
For example, the region labeled with $152$ belongs to the chamber of
$1\,3\,2=\tq{2}{3}\circ1\,2\,3$. Now, $152=124\circ\tq{2}{3}+e_2=142+010$. For another example,
the region labeled with $351$ belongs to the chamber of $3\,1\,2=\tq{2}{3}\circ\tq{1}{2}\circ1\,2\,3$, and
$124\circ\tq{1}{2}\circ\tq{2}{3}+e_1+e_2=351$.

Consider now e.g. $\bb=612$, which, as we said before, belongs to the orbit of $\ba=124$ (cf. Figure~\ref{2Catalan}), which has center
$\bz(\ba)=(1,2,2,3,3)$.
We can see that  also $\bz(\bb)=(1,2,2,3,3)$, being  $Z_0(\bb)=\{2\}$, $Z_1(\bb)=Z_2(\bb)=\{2,3\}$, $Z_3(\bb)=Z_4(\bb)=\{2,3,1\}$. Hence,
$\bb-\bp(\bb)=612-301=311$ and so, by Proposition~\ref{Kprop}, $I(\pi)=200$.
It follows that $\pi=2\,3\,1$ (cf. Figure~\ref{2Catalan}). But we can obtain this result in an easier, pictorial way, by using a representation similar
to that of Figure~\ref{Dycka} (cf. Figure~\ref{Dyck.inv}).
We come to this in Section~\ref{LDP.sec}
with the \emph{labeled Dyck paths}.
\end{example}

\begin{theorem}
In the $m$-Catalan arrangement, the center  is $\S_n$-invariant. More precisely, $\bz(\ba)=\bz(\bb)$ if and only if
$\ba$ and $\bb$ label two regions that belong to the same orbit under the action of $\S_n$.
\label{z.thm}
\end{theorem}
\begin{proof}
That if $\ba$ and $\bb$ label two regions that belong to the same orbit under the action of $\S_n$,
then $\bz(\ba)=\bz(\bb)$, follows immediately from Proposition~\ref{Kprop}.
For the other direction, consider the representation of the label $\ba=(a_1,\dotsc,a_n)\in\N^n$ of a \emph{region of the fundamental chamber}
as a Dyck path  through $(a_i,i-1)$ and $(a_i,i)$ for every $i\in[n]$, ending at
$(kn+1,n)$ and lying above $d$, as described before, and note that in general we may read the centers as indicated in
Figure~\ref{Dycka}.
In words, since $\ba$ is weakly increasing,
\begin{align}
&z_i(\ba)=\max\{j\in[n]\mid a_j\leq i+1\}\,.\label{fa2z.eq}\\
\shortintertext{On the opposite direction, this means that if $\bz=(z_0,\dotsc,z_{(m-1)n})$ and $\ba=(a_1,\dotsc,a_n)$ are weakly increasing and
$\bz=\bz(\ba)$, then}
&a_i=\min\{j\in[(m-1)n+1]\mid z_{j-1}\geq i\}\,.\label{fz2a.eq}\rlap{\qquad\qquad\qquad\qed\popQED}
\end{align}
\end{proof}

\section{$m$-Catalan functions}

We are now able to characterize the labels in the Pak-Stanley labeling of the $m$-Catalan arrangement.
\begin{definition}\label{def.cf}
An \emph{$m$-Catalan function} of dimension $n$ is $\ba=(a_1,\dotsc,a_n)\in\Z^n$
such that, for every $i\in[n]$,
$$z_{(i-1)m}(\ba)\geq i\,.$$
\end{definition}

\begin{lemma}
Let $\bz=\big(z_0,z_1,\dotsc,z_p\big)$ with $p\geq n$ and $z_0\leq z_1\leq\dotsb\leq z_p=n$. Then there
exist $n!$ different values of $\ba\in\N^n$ such that $(z_0,z_1,\dotsc,z_p)=\bz(\ba)$.
\label{ct.lemma}\end{lemma}
\begin{proof}
Let us see that we might choose $\ba$ such that $(z_0,z_1,\dotsc,z_p)=\bz(\ba)$ in $n!$ different ways.
First, we may choose
$Z_0\subseteq[n]$ such that $z_0=|Z_0|$ in $\binom{n}{z_0}$ different ways.
Let $Z_0=\{i_1,\dotsc,i_{z_0}\}$ with $i_1<\dotsb<i_{z_0}$. Then  $(a_{i_1},\dotsc,a_{i_{z_0}})$
may be chosen in $z_0!$ many ways such that $Z_0=Z_0(\ba)$, being any element of
$$[z_0]\times[z_0-1]\times\dotsb\times\{1,2\}\times\{1\}\,.$$
Now, let $k$ be the least integer such that
$z_k>z_0$;
we must choose the next set
$Z_k\supseteq Z_0$, which we may do in $\binom{n-z_0}{z_k-z_0}$ ways, and fill in the elements of this set in
$(z_k-z_0)!$ ways.
In fact, if $Z_k\setminus Z_0=\{i_1,i_2,\dotsc,i_\ell\}$ with
$i_\ell<i_{\ell-1}<\dotsb<i_1$ and $\ell=z_k-z_0$, and
$$t_j=\big|\{z\in Z_0\mid z>i_j  \big\}\big|\,,\ j=1,2,\dotsc,\ell$$
(and hence $t_1\leq t_2\leq\dotsb\leq t_\ell$) then, for having $Z_k=Z_k(\ba)$, we may still choose the values of $\ba$ of order $i_\ell,i_{\ell-1},\dotsc,i_1$ as any element of
\begin{equation}
\{k+t_\ell+1,\dotsc,k+t_\ell+\ell\}\times\dotsb\times\{k+t_2+1,k+t_2+2\}\times\{k+t_1+1\}\,.
\label{eq.fc}\end{equation}
Now, let $m$ be the least integer such that
$z_m>z_k$. Then, there are $\binom{n-z_k}{z_m-z_k}$ ways of choosing the next $(z_m-z_k)!$ elements, and so on.
Note that
\begin{align*}
&\binom{n}{z_0}z_0!\binom{n-z_0}{z_k-z_0}(z_k-z_0)!=\frac{n!}{(n-z_k)!}\\
\shortintertext{and}
&\frac{n!}{(n-z_k)!}\binom{n-z_k}{z_m-z_k}(z_m-z_k)!=\frac{n!}{(n-z_m)!}
\end{align*}
Hence, at the end there are $n!=\frac{n!}{(n-n)!}$   values of $\ba\in\N^n$ such that $\bz(\ba)=\bz$.
\end{proof}

\begin{theorem}[Main Theorem]
The labels in the Pak-Stanley labeling of the $m$-Catalan arrangement of dimension $n$ are exactly the
 $m$-Catalan functions of the same dimension. This labeling is bijective.
\end{theorem}

\begin{proof}
We begin by proving that the label of any region is an  $m$-Catalan function. Note that, by Theorem~\ref{z.thm}, it is sufficient to prove this result for the regions of the fundamental chamber, the weakly increasing functions
$\ba=(a_1,\dotsc,a_n)\in\N^n$ such that $a_i\leq 1+(i-1)m$ for every $i\in[n]$ (cf. \eqref{fchamber.eq}). But then
\begin{align*}
z_{(i-1)m}(\ba)&=\max\{j\in[n]\mid a_j\leq (i-1)m+1\}\\
&\geq i\,.\end{align*}
We now prove that a weakly increasing $m$-Catalan function $\ba=(a_1,\dotsc,a_n)$ is the label of a region of the
fundamental chamber.
We have seen that, if $a_1\leq\dotsb\leq a_n$ and
$\bz(\ba)=(z_0,\dotsc,z_{(m-1)n})$, then
\begin{align*}
&a_i=\min\{j\in[(m-1)n]\mid z_{j-1}\geq i\}\,.\\
\shortintertext{To say that}
&a_i\leq 1+(i-1)m\,,\ \text{for every $i\in[n]$}
\end{align*}
(cf. \eqref{fchamber.eq}) is then to say that
$$z_{(i-1)m}\geq i\,,$$
since $\bz$ is weakly increasing.

Finally, let $\bb=(b_1,\dotsc,b_n)$ be \emph{any} $m$-Catalan function with center
\begin{align*}
&\bz(\bb)=(z_0,\dotsc,z_{(m-1)n})
\shortintertext{and let $\ba=(a_1,\dotsc,a_n)$  be defined by}
&a_i=\min\{j\in[(m-1)n]\mid z_{j-1}\geq i\}\,,\ i=1,\dotsc,n
\end{align*}
(and then $a_1\leq a_2\leq\dotsb\leq a_n$). Since $\bz(\ba)=\bz(\bb)$ and hence $\ba$ is also an $m$-Catalan function,
it is the label of a region $R$ of the fundamental chamber. Now, by Lemma~\ref{ct.lemma} below and according to Remark~\ref{bij.rem},
$\bb$ is the label of some region in the orbit of $R$ under the action of $\S_n$.
\end{proof}

\subsection{Prime $m$-Catalan functions}
We have shown in an earlier work \cite[Proposition~4.11]{DGO} that the relatively bounded regions of the Shi arrangement are exactly the regions labeled with
\emph{prime parking functions}, which
 are the functions for which the unique weakly increasing rearrangement $(b_1,\dotsc,b_n)$ satisfies $b_{j+1}\leq j$ for every $j\in[n-1]$.
There is no known similar result for the $m$-Shi arrangement, where we may define the \emph{prime $m$-parking functions}
as  the functions for which the
unique weakly increasing rearrangement $(b_1,\dotsc,b_n)$ satisfies
 $$\forall\ j\in[n-1]\,,\ b_{j+1}\leq j\,m\,.$$

But this result is obvious if it is restricted to the fundamental chamber of the $m$-Shi arrangement
(or of any $(m,\ell)$-Coxeter arrangement), where
a ray from the origin crosses last a hyperplane of equation $x_j-x_{j+1}=m$.
Then, clearly a
point $P_{(x_1,\dotsc,x_n)}$ of the fundamental chamber
lies in a relatively bounded region if and only if
$$
\left\{
\arraycolsep=1pt
\begin{array}{cccccc}
&x_2+m&>&x_1&>&x_2\\
&x_3+m&>&x_2&>&x_3\\
&\vdots&&\vdots&&\vdots\\
&x_n+m&>&x_{n-1}&>&x_n
\end{array}\right.$$
Hence, fixed $1<j\leq n$, a relatively bounded region $R$ is separated from $R_0$ by at most $(j-2)m+m-1$ hyperplanes of equation
$x_i-x_j=a$ with $1\leq i<j$ and $a>0$. Thus, its label $\bb=(b_1,\dotsc,b_n)$ verifies $b_j\leq 1+(j-1)m-1$ for $j\geq2$, that is, $\bb$ is an increasing prime $m$-parking function.
On the other hand, if a hyperplane of equation $x_j=x_{j+1}+m$ separates a region $R$ of the fundamental chamber from $R_0$ (i.e., if $P_{(x_1,\dotsc,x_n)}\in R$, then $x_j>x_{j+1}+m$), then all hyperplanes of equation  $x_i=x_{j+1}+m$ with $i\leq j$ separate the same two regions. Hence,
$b_{j+1}\geq 1+j\,m$ and the label of $R$,
 $\bb$, is \emph{not} an increasing prime $m$-parking function.

 \begin{definition}\label{def.cf}
A \emph{prime} $m$-Catalan function of dimension $n$ is $\ba=(a_1,\dotsc,a_n)\in\Z^n$
such that, for every $i\in[2,n]$,
$$z_{(i-1)m-1}(\ba)\geq i\,.$$
\end{definition}

For example, in the $3$-dimensional $2$-Catalan arrangement,
$251$ is prime but $451$ is not. In fact, $\bz(251)=2\mathbf{2}2\mathbf{3}3$ but $\bz(451)=1\mathbf{1}2\mathbf{3}3$: whereas e.g.
$z_1(251)=2$,
$z_1(451)=1$.
\begin{proposition} In the $d$-dimensional $m$-Catalan arrangement, the \emph{prime} $m$-Catalan functions
are exactly the Pak-Stanley labels of the relatively bounded regions.
\end{proposition}

\begin{proof}
Note that every orbit of a region under the action of $\S_n$ either contains only relatively bounded regions or only
relatively unbounded regions. Hence, by Proposition~\ref{Kprop}, all we have to show is that the prime $m$-parking functions are exactly
the prime $m$-Catalan functions  \emph{in the fundamental chamber}.
We have seen before \eqref{fa2z.eq} that, given $\ba=(a_1,\dotsc,a_n)$ with $a_1\leq\dotsb\leq a_n$,
$z_i(\ba)=\max\{j\in[n]\mid a_j\leq i+1\}$.
Hence,
$$z_{(i-1)m-1}(\ba)=\max\{j\in[n]\mid a_j\leq (i-1)m\}$$
and thus, if $\ba$ is a prime $m$-parking function and  $a_i\leq (i-1)m$, then
$z_{(i-1)m-1}(\ba)\geq i$, and so $\ba$ is a prime $m$-Catalan function.

On the other hand, by \eqref{fz2a.eq},
$$a_{i+1}=\min\{j\in[(m-1)n+1]\mid z_{j-1}\geq i+1\}\,.$$
If $\ba$ is a prime $m$-Catalan function and so $z_{i\,m-1}(\ba)\geq i+1$, then
$$i\,m\in\{j\in[(m-1)n+1]\mid z_{j-1}\geq i+1\}$$
and so $a_{i+1}\leq i\,m$, which concludes the proof.
\end{proof}

\subsection{Labeled Dyck paths}\label{LDP.sec}
Starting with the center $\bz(\bb)$, and just by drawing a Dyck path going through\\
\noindent
\strut\hfill{\scriptsize$(0,0)$---$(1,0)$---$(3/2,z_0)$---$(5/2,z_1)$---$\,\dotsb\,$---$((m-1)n+3/2,z_{(m-1)n})$---$((m-1)n+3/2,n)$---$(mn+1,n)$}\hfill\strut

\noindent
we obtain a representation similar to that of Figure~\ref{Dycka}, where we can read easily $\ba$ such that $\bz=\bz(\ba)$ for the label
$\ba$ of a region of the fundamental chamber: at each level, we just take the first coordinate of the rightmost point of the path.
We may use such a drawing to see if $\bb$ is an $m$-Catalan function or if it is prime.

But things get more interesting if, at each level $i$,  the corner closest to the rightmost point of the path is filed with an element of
$Z_i(\bb)\setminus Z_{i-1}(\bb)$. If there is only one corner to fill at each level, then we obtain, written from bottom to top, $\pi$, the permutation to whose chamber belongs $\bb$. Two examples are presented in Figure~\ref{Dyck.inv}. Note that $612$ belongs to the orbit of $124$ and $513$ to the orbit of $133$, both on the chamber of $2\,3\,1$.
\begin{figure}[t]
\noindent
\strut\hfill
\begin{tikzpicture}[scale=.75]
\draw[step=1cm,gray, very thin] (0,0) grid (7,3);
\draw[very thin, red]  (0,2) -- (1,3);
\draw[very thin, red]  (0,1) -- (2,3);
\draw[very thin, red]  (0,0) -- (3,3);
\draw[very thin, red]  (2,1) -- (4,3);
\draw[very thin, red]  (4,2) -- (5,3);
\draw[thin] (0,0) -- (7,0) -- (7,3) -- (0,3) -- cycle;
\draw [fill] (1,0) node [above left, inner sep=0pt, minimum width=0pt, rectangle, fill=white] {\scriptsize$(1,0)$}
circle [radius=.0925] -- (.5,.12);
\draw [fill] (1,1) node {} circle [radius=.0225];
\draw [fill] (2,1) node [above left] {} circle [radius=.0925];
\draw [fill] (2,2) node {} circle [radius=.0225];
\draw [fill] (4,2) node [above left] {} circle [radius=.0925];
\draw [fill] (4,3) node {} circle [radius=.0225];
\draw [fill] (7,3) node [above] {\scriptsize$(7,3)$} circle [radius=.0725] ;
\draw [very thick]  (0,0) -- (1,0) -- (1,1) -- (2,1) -- (2,2) -- (4,2) -- (4,3) -- (7,3);
\draw  (1.5,0.5) node  {$2$};
\draw  (2.5,1.5) node [rectangle, fill=white] {$1$};
\draw  (4.5,2.5) node [rectangle, fill=white] {$3$};
\draw [ultra thick, densely dashed, red] (1,0) -- (1,1) -- (3,1) -- (3,2) -- (5,2) -- (5,3) --
node [pos=.75, above left] {\scriptsize\rlap{$\downarrow$}$\ 2$-Catalan} (7,3);
\draw  (1.5,0.5) node  {$2$};
\draw  (2.5,0.5) node  {\color{gray}$2$};
\draw  (3.5,0.5) node  {\color{gray}$2$};
\draw  (4.5,0.5) node  {\color{gray}$2$};
\draw  (5.5,0.5) node  {\color{gray}$2$};
\draw  (6.5,0.5) node  {\color{gray}$2$};
\draw  (2.5,1.5) node [rectangle, fill=white, inner sep=0pt, minimum width=0pt] {$3$};
\draw  (3.5,1.5) node  {\color{gray}$3$};
\draw  (4.5,1.5) node  {\color{gray}$3$};
\draw  (5.5,1.5) node  {\color{gray}$3$};
\draw  (6.5,1.5) node  {\color{gray}$3$};
\draw  (4.5,2.5) node [rectangle, fill=white, inner sep=0pt, minimum width=0pt] {$1$};
\draw  (5.5,2.5) node  {\color{gray}$1$};
\draw  (6.5,2.5) node  {\color{gray}$1$};
\draw  (1.5,-0.5) node  {$Z_0$};
\draw  (2.5,-0.5) node  {$Z_1$};
\draw  (3.5,-0.5) node  {$Z_2$};
\draw  (4.5,-0.5) node  {$Z_3$};
\draw  (5.5,-0.5) node  {$Z_4$};
\draw  (6.5,-0.5) node  {$Z_5$};
\end{tikzpicture}
\hfill
\begin{tikzpicture}[scale=.75]
\draw[step=1cm,gray, very thin] (0,0) grid (7,3);
\draw[very thin, blue]  (0,2) -- (1,3);
\draw[very thin, blue]  (0,1) -- (2,3);
\draw[very thin, blue]  (0,0) -- (3,3);
\draw[very thin, blue]  (3,2) -- (4,3);
\draw[very thin, red]  (2,1) -- (3,2);
\draw[very thin, red]  (4,2) -- (5,3);
\draw[thin] (0,0) -- (7,0) -- (7,3) -- (0,3) -- cycle;
\draw [fill] (1,0) node [above left, inner sep=0pt, minimum width=0pt, rectangle, fill=white] {\scriptsize$(1,0)$}
circle [radius=.0925] -- (.5,.12);
\draw [fill] (1,1) node {} circle [radius=.0225];
\draw [fill] (2,1) node {} circle [radius=.0225];
\draw [fill] (3,1) node {} circle [radius=.0925];
\draw [fill] (3,2) node {} circle [radius=.0925];
\draw [fill] (4,2) node {} circle [radius=.0225];
\draw [fill] (7,3) node [above] {\scriptsize$(7,3)$} circle [radius=.0725];
\draw [very thick]  (0,0) -- (1,0) -- (1,1) -- (2,1) -- (3,1) -- (3,2) -- (3,3) --
node [pos=.9, above left] {\red\scriptsize\rlap{$\downarrow$}$\ 2$-Catalan} (7,3);
\draw [ultra thick, dashed, red] (1,0) -- (1,1) -- (3,1) -- (3,2) -- (5,2) -- (5,3) -- (7,3);
\draw [ultra thick, dashed, blue]  (2,1) --  (2,2)  -- node [rectangle, fill=white, pos=0.1, above, inner sep=1pt, minimum width=1pt] {\footnotesize prime}  (4,2) -- (4,3);
\draw  (1.5,0.5) node  {$2$};
\draw  (2.5,0.5) node  {\color{gray}$2$};
\draw  (3.5,0.5) node  {\color{gray}$2$};
\draw  (4.5,0.5) node  {\color{gray}$2$};
\draw  (5.5,0.5) node  {\color{gray}$2$};
\draw  (6.5,0.5) node  {\color{gray}$2$};
\draw  (3.5,1.5) node  {$3$};
\draw  (4.5,1.5) node  {\color{gray}$3$};
\draw  (5.5,1.5) node  {\color{gray}$3$};
\draw  (6.5,1.5) node  {\color{gray}$3$};
\draw  (3.5,2.5) node [rectangle, fill=white, inner sep=0pt, minimum width=0pt]  {$1$};
\draw  (4.5,2.5) node [rectangle, fill=white, inner sep=0pt, minimum width=0pt]  {\color{gray}$1$};
\draw  (5.5,2.5) node  {\color{gray}$1$};
\draw  (6.5,2.5) node  {\color{gray}$1$};
\draw  (1.5,-0.5) node  {$Z_0$};
\draw  (2.5,-0.5) node  {$Z_1$};
\draw  (3.5,-0.5) node  {$Z_2$};
\draw  (4.5,-0.5) node  {$Z_3$};
\draw  (5.5,-0.5) node  {$Z_4$};
\draw  (6.5,-0.5) node  {$Z_5$};
\end{tikzpicture}\hfill\strut
\caption{Labeled Dyck path representation of $612$ (left) and of $513$ (right).
Note that, for example,  $\bz(612)=12233$ and  $Z_2(612)=\{2,3\}$.}
The former is on the orbit of $124$ and the latter on the orbit of $133$. Both belong to the chamber of  $2\,3\,1$(cf. Figure~\ref{2Catalan})
\label{Dyck.inv}
\end{figure}
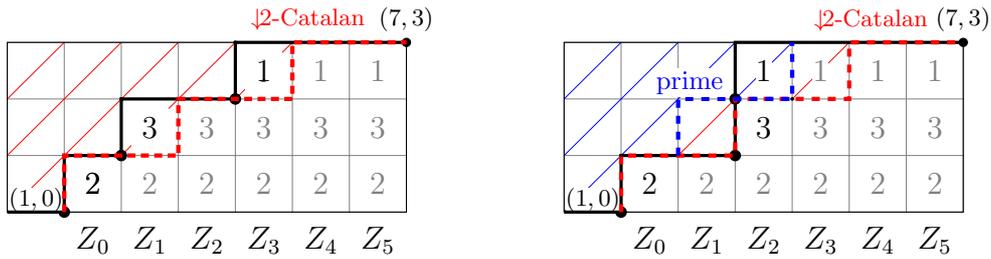

We present a last example in Figure~\ref{Dyck3}, where we represent the centers of $\bb=39\,551\,481$.
Note that $Z_0(\bb)=\{1,5,8\}=Z_1(\bb)$ and $Z_2(\bb)=\{1,3,4,5,6,8\}$. To order the ``towers'' in a convenient way, note that, by Proposition~\ref{Kprop},
$I(\pi)=\bb-\bp(\bb)-\su=39\,551\,481-03\,220\,260-11\,111\,111=25\,220\,110$.
Now, it is easy to see that $\pi=5\,8\,1\,3\,4\,6\,2\,7$.  The label of the region of the fundamental chamber that
belongs to the orbit of $\bb$ under the action of $\S_n$ is $\ba=11\,133\,347$ and can be directly read from the picture.

\begin{figure}[th]
\noindent
\begin{tikzpicture}[scale=.5]
\draw[step=1cm,gray, very thin] (0,0) grid (13,8);
\draw[thin] (0,0) -- (13,0) -- (13,8) -- (0,8) -- cycle;
\draw [fill] (1,0) node {} circle [radius=.0925];
\draw [fill] (1,1) node {} circle [radius=.0925];
\draw [fill] (1,2) node {} circle [radius=.0925];
\draw [fill] (1,3) node {} circle [radius=.0225];
\draw [fill] (2,3) node {} circle [radius=.0225];
\draw [fill] (3,3) node {} circle [radius=.0925];
\draw [fill] (3,4) node {} circle [radius=.0925];
\draw [fill] (3,5) node {} circle [radius=.0925];
\draw [fill] (3,6) node {} circle [radius=.0225];
\draw [fill] (4,6) node {} circle [radius=.0925];
\draw [fill] (4,7) node {} circle [radius=.0225];
\draw [fill] (5,7) node {} circle [radius=.0225];
\draw [fill] (6,7) node {} circle [radius=.0225];
\draw [fill] (7,7) node {} circle [radius=.0925];
\draw [fill] (7,8) node {} circle [radius=.0225];
\draw [fill] (8,8) node {} circle [radius=.0225];
\draw [fill] (9,8) node {} circle [radius=.0225];
\draw [fill] (13,8) node {} circle [radius=.0225];
\draw [ultra thick]  (0,0) -- (1,0) -- (1,2) -- (1,3) -- (2,3) -- (3,3) -- (3,4) -- (3,5) -- (3,6) -- (4,6) -- (4,7)
-- (5,7) -- (6,7) -- (7,7) -- (7,8) -- (9,8) -- (10.25,8) -- (10.5,8.25) -- (10.75,7.75) -- (11.25,8.25) -- (11.5,7.75) -- (11.75,8) -- (13,8) ;
\draw  (1.5,0.5) node  {\normalsize$5$};
\draw  (1.5,1.5) node  {\normalsize$8$};
\draw  (1.5,2.5) node  {\normalsize$1$};
\draw  (3.5,3.5) node  {\normalsize$3$};
\draw  (3.5,4.5) node  {\normalsize$4$};
\draw  (3.5,5.5) node  {\normalsize$6$};
\draw  (4.5,6.5) node  {\normalsize$2$};
\draw  (7.5,7.5) node  {\normalsize$7$};
\end{tikzpicture}
\caption{Labeled Dyck path for the $3$-Catalan function $\bb=39\,551\,481$}
\label{Dyck3}
\end{figure}
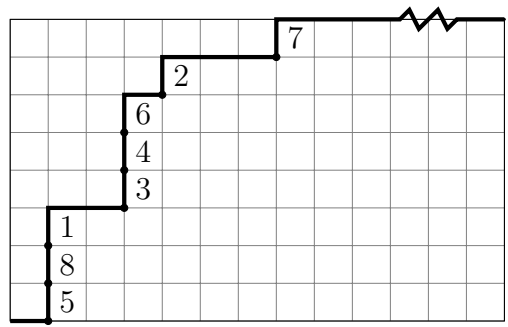

\end{document}